\documentclass[]{amsart}
\usepackage[utf8]{inputenc}
\usepackage{amsmath}
\usepackage{amsfonts}
\usepackage{amsthm}
\usepackage{amssymb}
\usepackage{color}
\usepackage{todonotes}
\usepackage{comment}
\usepackage{tikz}
\usepackage{hvfloat}
\usepackage{enumerate}
\usepackage{todonotes}
\usetikzlibrary{decorations.markings}
\usetikzlibrary{arrows}
\usepackage{graphicx}
\usepackage{caption}
\usepackage{subcaption}
\usepackage{mathtools}
\usepackage[all]{xy}

\newcommand{\Z}{\mathbb{Z}}

\newcommand{\N}{\mathbb{N}}

\newcommand{\B}{\mathbb{B}}

\newcommand{\grate}{\mathbf{h}}

\newcommand\prim{\mathsf{Prim}}
\newcommand\cycrep{\mathsf{CycRep}}
\newcommand\cycsl{\mathsf{CycSL}}

\newcommand\sphl{\mathsf{SL}} 
 
\newcommand\supp{\mathsf{Supp}}

\renewcommand{\le}{\leqslant}
\renewcommand{\ge}{\geqslant}

\theoremstyle{plain}
\newtheorem{thm}{Theorem}[section]
\newtheorem{lem}[thm]{Lemma}
\newtheorem{prop}[thm]{Proposition}
\newtheorem{cor}[thm]{Corollary}

\newtheorem{quR}{Question}

\makeatletter
\newtheorem*{rep@theorem}{\rep@title}
\newcommand{\newreptheorem}[2]{%
\newenvironment{rep#1}[1]{%
 \def\rep@title{#2 \ref{##1}}%
 \begin{rep@theorem}}%
 {\end{rep@theorem}}}
\makeatother
\newreptheorem{theorem}{Theorem}

\newtheorem*{thm*}{Theorem}
\newtheorem*{lem*}{Lemma}
\newtheorem*{prop*}{Proposition}
\newtheorem*{cor*}{Corollary}
\newtheorem*{qu*}{Question}
\newtheorem*{dt*}{Definition and Theorem}
\newtheorem*{not*}{Notation}
\newtheorem*{exmp*}{Example}
\newtheorem*{exmps*}{Examples}
\newtheorem*{dprop*}{Definition and Proposition}
\newtheorem*{conj*}{Conjecture}

\theoremstyle{definition}
\newtheorem{defn}[thm]{Definition}
\newtheorem{exmp}[thm]{Example}

\newtheorem*{defn*}{Definition}
\newtheorem{conj}[thm]{Conjecture}

\theoremstyle{plain}
\newtheorem{rem}[thm]{Remark}
\newtheorem*{rem*}{Remark}

\DeclareMathOperator\dc{dc}

\DeclareMathOperator\CR{cr}

\begin{document}
\title{The conjugacy ratio of groups}

\author{Laura Ciobanu}
\address{Heriot-Watt University, Edinburgh, EH14 4AS, UK}
\email{l.ciobanu@hw.ac.uk}
\urladdr{http://www.macs.hw.ac.uk/~lc45/}

\author{Charles Garnet Cox}
\address{University of Bath, BA2 7AY, UK}
\email{cpgcox@gmail.com}

\author{Armando Martino}
\address{Mathematical Sciences, University of Southampton, SO17 1BJ, UK}
\email{A.Martino@soton.ac.uk}
\urladdr{http://www.personal.soton.ac.uk/am1t07/}

\thanks{}


\date{\today}
\begin{abstract}
In this paper we introduce and study the conjugacy ratio of a finitely generated group, which is the limit at infinity of the quotient of the conjugacy and standard growth functions. We conjecture that the conjugacy ratio is $0$ for all groups except the virtually abelian ones, and confirm this conjecture for certain residually finite groups of subexponential growth, hyperbolic groups, right-angled Artin groups, and the lamplighter group.
\end{abstract}
\maketitle

\section{Introduction}
In this paper we introduce and study the conjugacy ratio of a group, which is the limit of the quotient of two functions naturally associated to any finitely generated group: conjugacy growth and standard growth. More precisely, if $G$ is generated by the finite set $X$, let $\B_{G,X}(n)$ denote the ball of radius $n$ with respect to $X$, and let $C_{G,X}(n)$ denote the set of conjugacy classes of $G$ which have a representative in $\B_{G,X}(n)$. Then the {\em conjugacy ratio} of $G$ with respect to $X$ is:
\begin{align}\label{dcalt}
\CR_X(G) = \limsup_{n\rightarrow\infty} \frac{|C_{G,X}(n)|}{|\B_{G,X}(n)|}.
\end{align} 

The motivation of this paper is twofold. On one hand, the conjugacy ratio of a finite group $H$ is equal to the degree of commutativity $\dc(H)$ of $H$, which measures the probability that two elements of the group commute, and is defined as:
\begin{align}\label{dcfinite}
\dc(H)=\frac{|\{ (x,y) \in H \times H \ : \ xy=yx \}|}{|H|^2}.
\end{align}
The degree of commutativity of a group has received a lot of attention recently, as its definition was extended to finitely generated infinite groups in \cite{dcA} to be
\begin{align*}
\dc_X(G)=\limsup_{n\rightarrow\infty}\frac{|\{(x, y) \in \B_{G,X}(n)^2 :  ab=ba\}|}{|\B_{G,X}(n)|^2}.
\end{align*}
As raised in \cite{Cox3}, it is natural to explore whether the degree of commutativity and the conjugacy ratio are related for infinite groups as well. 

Our second motivation comes from the fact that very few quantitative results comparing standard and conjugacy growth in groups exist in the literature. While in any group there are fewer conjugacy classes than elements, the gap between these two functions has not been been explored in detail, and it is worth investigating. For example, the standard and conjugacy growth rates (i.e. taking the limit of the $n$th root of the function at $n$) are equal in some of the most frequently encountered families of infinite groups: hyperbolic groups \cite{AC17}, graph products \cite{CHM17}, many wreath products \cite{VM17}; thus in these examples the quotient of the two functions, as a function of $n$, must be at most subexponential, and if the conjugacy ratio is $0$, the convergence to $0$ will not be very fast.

Our starting point is the following conjecture, inspired by \cite[Conj. 1.6]{dcA}. 
\begin{conj}
\label{maincon}
Let $G$ be a group generated by a finite set $X$. Then $\CR_X(G) > 0$ if and only if $G$ is virtually abelian.
\end{conj}

Our results on the conjugacy ratio in several families of groups support Conjecture \ref{maincon}. In Section \ref{subexp} we investigate groups of stable subexponential growth (Definition \ref{stable}). We first show that any virtually abelian group has $\CR_X(G)>0$ for any finite generating set $X$. We then show that, if $N$ is a normal, finite index subgroup of $G$, then (for any finite generating set $X$ of $G$) $\CR_X(G)\le \dc(G/N)$. This allows us to apply a technique from \cite{dcA} to show that any residually finite group $G$ of stable subexponential growth which is not virtually abelian has $\CR_X(G)=0$ for any finite generating set $X$.
We also show in Theorem~\ref{indx} that if $G$ is a finitely generated virtually abelian group, with finite generating sets $X$ and $Y$, then $\CR_X(G) = \CR_Y(G)$.

We say that a group, $G$, with generating set $X$ has stable subexponential growth if $\lim_{n \to \infty} \frac{|\B_{G, X}(n+1)|}{|\B_{G, X}(n)|}=1$, Definition~\ref{stable}. This includes all finitely generated virtually-nilpotent groups. Since all finitely generated virtually-nilpotent groups are residually finite, the theorem below means that Conjecture \ref{maincon} is true for all groups of polynomial growth.

\begin{reptheorem}{thmsubsexp}
The conjugacy ratio for all finitely generated, residually finite groups of stable subexponential growth that are not virtually abelian is zero, with respect to all finite generating sets. 
\end{reptheorem}
The proof of Theorem \ref{thmsubsexp} cannot be generalised to groups of exponential growth, but we provide independent arguments for several important classes of groups of exponential growth.
\begin{reptheorem}{thm:hyperbolic}
Let $G$ be a non-elementary hyperbolic group. Then $\CR_X(G)= 0$ for any finite generating set $X$.
\end{reptheorem}
 
\begin{reptheorem}{thm:lamplighter}
Let $G$ be the lamplighter group, that is, the wreath product $C_2\wr\mathbb{Z}$. Then $\CR_X(G)= 0$ for the standard generating set $X$ (defined in (\ref{vecY})).
\end{reptheorem}

\begin{reptheorem}{thm:RAAG}
Let $G=(G_V, X_V)$ be a right-angled Artin group (RAAG) based on a graph $\Gamma=(V,E)$ with generating set $X_V$. Then $\CR_{X_V}(G)= 0$ unless $G$ is free abelian, in which case $\CR_{X_V}(G)=1$.
\end{reptheorem}

We may also consider the \textit{strict} or \textit{spherical} conjugacy ratio, where the counting is done in the sphere of radius $n$ rather than the ball of radius $n$, that is, we may take the ratio of the strict conjugacy growth function over the spherical growth function. More precisely, let $S_{G,X}(n)$ be the sphere of radius $n$ in the group $G$ with respect to finite generating set $X$, and let $C^s_{G,X}(n)$ be the  conjugacy classes that intersect $S_{G,X}(n)$ but not $\B_{G,X}(n-1)$, that is, those conjugacy classes with a minimal length representative in $S_{G,X}(n)$.  The spherical conjugacy ratio is then
\begin{align}\label{dcalt}
\CR^s_X(G) = \limsup_{n\rightarrow\infty} \frac{|C^s_{G,X}(n)|}{|S_{G,X}(n)|}.
\end{align} 

\begin{rem}\label{SC} By the Stolz-Ces\`aro theorem, anytime the spherical conjugacy ratio turns out to be a limit, the conjugacy ratio will be equal to this limit. In particular, if the spherical conjugacy ratio is $0$, then the conjugacy ratio is $0$.
\end{rem}



\section{Preliminaries}


Recall that for a finitely generated group, $G$, with generating set $X$, the exponential growth rate of $G$ with respect to $X$ is:
\begin{equation}\label{growth_rate}
Exp_X(G) = \lim_{n\rightarrow\infty} \sqrt[n]{ |\B_{G,X}(n)|}.
\end{equation}

\begin{defn}
	A group, $G$, with finite generating set $X$, is said to have exponential growth if $Exp_X(G) > 1$ and subexponential growth if $Exp_X(G) = 1$. This does not depend on the generating set, $X$.
\end{defn}

Additionally, for any $\epsilon > 0$, if $\lambda = Exp_X(G)$, then for sufficiently large $n$, 
$$
\lambda^n \leq |\B_{G,X}(n)| \leq (\lambda + \epsilon)^n. 
$$
Moreover, if we replace balls with spheres, we get the same limit and inequality.

We collect below a few results on convergence of series that will be relevant later.

\begin{thm}[Stolz-Ces\`aro]\label{SC}
Let $a_n, b_n$, $n\geq 1$ be two sequences with $b_n$ strictly increasing and divergent. If the lefthandside limit exists, $$\lim_{n\rightarrow\infty} \frac{a_{n+1}-a_n}{b_{n+1}-b_n}= l \implies \lim_{n\rightarrow\infty} \frac{a_n}{b_n}= l.$$
\end{thm}

Proposition \ref{SCconverse} is a partial converse to the Stolz-Ces\`aro theorem. It implies that for groups of exponential growth, if the conjugacy ratio is a limit and the ratio of sizes of consecutive balls has a limit, then the spherical conjugacy ratio is equal to the conjugacy ratio. 

\begin{prop}\label{SCconverse}
	Let $a_n, b_n$, $n\geq 1$ be two sequences with $b_n$ strictly increasing and divergent, such that the lefthandside limit exists and $\lim_{n\rightarrow\infty} \frac{b_{n+1}}{b_n} \neq 1$. Then $$\lim_{n\rightarrow\infty} \frac{a_n}{b_n}= l \implies \lim_{n\rightarrow\infty} \frac{a_{n+1}-a_n}{b_{n+1}-b_n}= l.$$
\end{prop}

\begin{prop}\label{convolution}
Let $a_n, b_n, c_n, d_n$, $n\geq 0$ be monotonically increasing sequences of positive integers.
Define the sequences $\widehat{c}_n$ and $\widehat{d}_n$ as $\widehat{c}_0:=c_0$, $\widehat{d}_0:=d_0$, and $\widehat{c}_n:=c_n - c_{n-1}$ and $\widehat{d}_n:=d_n - d_{n-1}$, for $n \geq 1$. 

Suppose that
\begin{enumerate}[(i)]
	\item $a_n \leq b_n$ and $\widehat{c}_n \leq \widehat{d}_n$ for all $n$, 
	\item $\frac{a_n}{b_n} \to 0$ and $\frac{c_n}{d_n} \to 0$ as $n \to \infty$.
\end{enumerate}
Then 
$$\lim_{n\rightarrow\infty} \frac{\sum_{i=0}^n a_i \widehat{c}_{n-i}}{\sum_{i=0}^n b_i \widehat{d}_{n-i}}=0.$$
\end{prop}
\begin{proof}

Given $\epsilon > 0$, fix an $N$ such that $\frac{a_n}{b_n}  < \epsilon$ for all $n \geq N$. Next choose an $M \geq N$ such that $\frac{c_n}{d_n} < \frac{\epsilon}{a_N}$ for all $n \geq M$.

Then, for $n \geq M \geq N$,
$$
\sum_{i=N}^n a_i \widehat{c}_{n-i} < \epsilon \sum_{i=N}^n b_i \widehat{c}_{n-i} \leq \epsilon \sum_{i=0}^n b_i \widehat{d}_{n-i}.
$$
Thus, for $n\geq M$,
$$
\frac{\sum_{i=0}^n a_i \widehat{c}_{n-i}}{\sum_{i=0}^n b_i \widehat{d}_{n-i}}=\frac{\sum_{i=0}^N a_i \widehat{c}_{n-i}}{\sum_{i=0}^n b_i \widehat{d}_{n-i}}+\frac{\sum_{i=N+1}^n a_i \widehat{c}_{n-i}}{\sum_{i=0}^n b_i \widehat{d}_{n-i}}<\frac{\sum_{i=0}^N a_i \widehat{c}_{n-i}}{\sum_{i=0}^n b_i \widehat{d}_{n-i}}+\epsilon.$$
Now we obtain the result by using the fact that for $n \geq M$
$$
\frac{\sum_{i=0}^N a_i \widehat{c}_{n-i}}{\sum_{i=0}^n b_i \widehat{d}_{n-i}} \leq a_N 
\frac{\sum_{i=0}^N \widehat{c}_{n-i}}{\sum_{i=0}^n \widehat{d}_{n-i}} \leq a_N
\frac{c_n}{d_n} < \epsilon.
$$
	
\end{proof}

\begin{prop}
	\label{convolution2}
Let $a_n, b_n, c_n, d_n$, $n \geq 0$, be sequences of positive integers satisfying the following properties:
\begin{enumerate}[(i)]
	\item $a_n, b_n$ are monotone sequences,
	\item $a_n \leq b_n$ and $c_n \leq d_n$ for all $n$, 
	\item $\frac{a_n}{b_n} \to 0$ as $n \to \infty$,
	\item $\frac{d_n}{b_n} \leq \delta^n$ for all sufficiently large $n$, and for some $0 < \delta < 1$.
\end{enumerate}
Then, 
$$\lim_{n\rightarrow\infty} \frac{\sum_{i=0}^n a_i c_{n-i}}{\sum_{i=0}^n b_i d_{n-i}}=0.$$
\end{prop}
\begin{proof}
	Given $\epsilon > 0$, fix an $N$ such that $\frac{a_n}{b_n}  < \epsilon'<\epsilon$ for all $n \geq N$. Then, for $n \geq N$,
	$$
     \sum_{i=N}^n a_i c_{n-i} < \epsilon' \sum_{i=N}^n b_i c_{n-i} \leq \epsilon' \sum_{i=0}^n b_i d_{n-i}.
	$$
	Thus, for $n\geq N$,
		$$
\frac{\sum_{i=0}^n a_i c_{n-i}}{\sum_{i=0}^n b_i d_{n-i}}=\frac{\sum_{i=0}^N a_i c_{n-i}}{\sum_{i=0}^n b_i d_{n-i}}+\frac{\sum_{i=N+1}^n a_i c_{n-i}}{\sum_{i=0}^n b_i d_{n-i}}<\frac{\sum_{i=0}^N a_i c_{n-i}}{\sum_{i=0}^n b_i d_{n-i}}+\epsilon'$$
		and so it suffices to show that
		$$\lim_{n\rightarrow\infty}\frac{\sum_{i=0}^N a_i c_{n-i}}{\sum_{i=0}^n b_i d_{n-i}} \leq \lim_{n\rightarrow\infty}\frac{\sum_{i=0}^N a_i c_{n-i}}{b_n d_0}=0.$$	
	
	Now
	$$\frac{\sum_{i=0}^N a_i c_{n-i}}{b_n d_0}\le\frac{\sum_{i=0}^N a_i c_{n-i}}{b_n} \leq a_N \sum_{i=0}^N \frac{c_{n-i}}{b_{n-i}} \leq a_N \sum_{i=0}^N \frac{d_{n-i}}{b_{n-i}}.
$$
Using hypothesis (iv), there is a sufficiently large $n$ such that
$$
a_N \sum_{i=0}^N \frac{d_{n-i}}{b_{n-i}} \leq a_N \sum_{i=0}^N \delta^{n-i}=a_N\frac{\delta^n}{\delta^N} \left(\frac{1-\delta^{N+1}}{1-\delta} \right) \leq \delta^n \left( \frac{a_N}{\delta^N} \frac{1}{(1-\delta)} \right)<\epsilon-\epsilon'.\qedhere 
$$
\end{proof}


%
%

\section{Results for groups of stable subexponential growth}\label{subexp}

\begin{defn}\label{stable}
	A group $G$, with finite generating set, $X$, is said to be of {\em stable subexponential growth} if $\lim_{n\rightarrow\infty}\frac{|\B_{G.X}(n+1)|}{|\B_{G,X}(n)|}=1$.
\end{defn}

Note that being of stable subexponential growth, implies that $Exp_X(G) = 1$, and hence that the group has subexponential growth. 

\medskip

By the celebrated result of Gromov, every finitely generated group of polynomial growth - where $\B_{G, X}(n)$ is bounded above by a polynomial function - is virtually nilpotent. All these groups are of stable subexponential growth since, by a result of Bass, \cite{bass}, if $G$ is a finitely generated, virtually nilpotent group, and $X$ is any finite generating set, then, for some exponent $d$, and constants, $A,B$:
\begin{equation}\label{eqn:pol_bounds}
A n^d \leq  |\B_{G,X}(n)| \leq B n^d.
\end{equation}
The exponent $d$ is calculated explicitly in \cite{bass}; for a virtually abelian group it is equal to the rank of a finite index free abelian subgroup. 

From (\ref{eqn:pol_bounds}) we get that for any positive integer, $k$, 
\begin{equation}\label{eqn:ball_ratio}
\lim_{n \to \infty} \frac{|\B_{G,X}(n+k)|}{|\B_{G,X}(n)|} = 1.
\end{equation}

The main result which we require for this class is the following.

\begin{prop} \cite{indexformula}. \label{pepprop}Let $G$ be a finitely generated group with stable subexponential growth, and finite generating set $X$. For every finite index subgroup $H\leqslant G$ and every $g \in G$, we have
$$\lim_{n\rightarrow\infty}\frac{|gH\cap\B_{G,X}(n)|}{|\B_{G,X}(n)|}=\lim_{n\rightarrow\infty}\frac{|Hg\cap\B_{G,X}(n)|}{|\B_{G,X}(n)|}=\frac{1}{[G:H]}.$$

Furthermore, if $H$ is an infinite index subgroup of $G$ then both limits are zero for any coset of $H$. 
\end{prop}

\begin{rem}
	The last statement does not appear explicitly in \cite{indexformula}, but follows easily from their arguments. Alternatively, one could prove this via the construction of an invariant mean which requires the choice of an ultrafilter. The stable subexponential condition ensures that any ultrafilter will do, and hence that all limit points of the sequences above are equal.
\end{rem}

From now on, whenever there is no ambiguity concerning the group and its generating set, we will write $C(n)$ instead of $C_{G,X}(n)$ and $\B(n)$ instead of $\B_{G, X}(n)$.

\begin{prop}\label{cr>0}
Suppose that $G$ is a finitely generated, virtually abelian group. Then, for any finite generating set $X$ of $G$, we have that $\CR_X(G) >0$. 

More precisely, if $[G:A]=m$ where $A$ is abelian, then $\CR_X(G) \geq 1/m^2$.  
\end{prop}
\begin{proof} Let $[G:A]=m$, where $A$ is abelian. We note that $G$ acts by multiplication on the right cosets of $A$. If $g$ and $h$ lie in the same right coset, then $h=\alpha g$ for some $\alpha \in A$, so for any $a \in A$, $h^{-1}ah=(\alpha g)^{-1}a(\alpha g)=g^{-1}ag$ since $A$ is abelian. Thus there are at most $m$ conjugates of each element $a \in A$ and so, for all $n \in \N$, we have that $|C(n)\cap A|\ge |\B(n)\cap A|\cdot \frac{1}{m}$. Now
\begin{align*}
\frac{|C(n)|}{|\B(n)|}\ge\frac{|C(n)\cap A|}{|\B(n)|}=&\frac{|\B(n)\cap A|}{|\B(n)}\cdot\frac{|C(n)\cap A|}{|\B(n)\cap A|}
\ge\frac{|\B(n)\cap A|}{|\B(n)|}\cdot\frac{1}{m}
\end{align*}
which tends to $1/m^2$ by Proposition \ref{pepprop}.
\end{proof}

\begin{lem} Let $G$ be a group of stable subexponential growth with finite generating set $X$, let $g \in G$ and let $H$ be a finite index subgroup of $G$. For $d \in \N$ we have
$$\lim_{n\rightarrow\infty}\frac{|gH\cap\B_{G,X}(n+d)|}{|\B_{G,X}(n)|}=\frac{1}{[G:H]}.$$
\end{lem}
\begin{proof} This follows from writing
$$\lim_{n\rightarrow\infty}\frac{|gH\cap\B(n+d)|}{|\B(n)|}=\lim_{n\rightarrow\infty}\frac{|\B(n+d)|}{|\B(n)|}\frac{|gH\cap\B(n+d)|}{|\B(n+d)|}$$
together with Proposition \ref{pepprop} and (\ref{eqn:ball_ratio}).

\end{proof}

\begin{prop}\label{subexpprop}
Let $G$ be a finitely generated group of stable subexponential growth and $N$ a subgroup of finite index in $G$. Then $\CR_X(G) \leq dc(G/N)$ for any finite generating set $X$ of $G$.
\end{prop}
\begin{proof}
Let $[G:N]=m$, so that $G=g_1N\sqcup g_2N\sqcup\ldots\sqcup g_mN$ for some $g_1, \ldots, g_m \in G$. Let $d:=\max\{|g_i|_X\;:\;i=1, \ldots, m\}$.

Now consider if $xN\sim yN$ (in $G/N$). Then $yN=g^{-1}xgN=g^{-1}xgg^{-1}Ng=g^{-1}xNg$ for some $g \in G$. Moreover, since $x$ and $y$ are conjugate in $G/N$, we may choose $g$ from $\{g_1, \ldots, g_m\}$ and so $|g|_X\le d$. Now let $yk_1 \in \B(n)$. We know there must exist some $xk_2 \in xN$ such that $g^{-1}(xk_2)g=yk_1$. But then $xk_2=gyk_1g^{-1}$, and so $xk_2 \in \B(n+2d)$. Hence, for every $n \in \N$, each element in $\B(n)\cap yN$ is conjugate to some element in $\B(n+2d)\cap xN$.

Let $x_1, \ldots x_k \in \{g_1, \ldots, g_m\}$ be the representatives of the conjugacy classes in $G/N$. For every $i \in \N$ and every $j \in \Z_k$, we will assume that there are $|\B(n)\cap x_jN|$ conjugacy classes in $\B(n)\cap x_jN$. Hence $$\frac{|C(n)|}{|\B(n)|}\le \frac{\sum_{i=1}^k|x_iN\cap \B(n+2d)|}{|\B(n)|}$$ which tends to $k/m$ by the previous lemma.
\end{proof}

\begin{thm}\label{thmsubsexp}
Conjecture \ref{maincon} is true for all finitely generated, residually finite groups of stable subexponetial growth. 
\end{thm}
\begin{proof} Proposition \ref{cr>0} states that, if a finitely generated group $G$ is virtually abelian, then, for any finite generating set $X$, $\CR_X(G)>0$. For the other direction we apply the method of \cite[Proof of Thm. 1.3]{dcA} by using Proposition \ref{subexpprop}. For completeness we will describe their argument. It requires the following result from \cite{gallagher}: if $F$ is a finite group and $N \unlhd F$, then
\begin{equation}\label{gallagher}
\dc(F)\le \dc(F/N)\cdot \dc(N).
\end{equation}
Our hypotheses are that $G$ is: finitely generated, residually finite, of stable subexponential growth, and not virtually abelian. We wish to show that $\CR_X(G)=0$ for any finite generating set $X$. We will work with finite quotients and will build a chain of normal subgroups. Since $G$ is finitely generated we may choose these subgroups to be characteristic, and will do this because being characteristic is transitive.

Since $G$ is not virtually abelian, choose $g_1, g_2 \in G$ that do not commute and, using the residually finite assumption, let $[g_1, g_2] \not\in K_1$ where $K_1$ is a characteristic and finite index subgroup of $G$. Hence $G/K_1$ is non-abelian, and by Gustafson's result we have that $\dc(G/K_1)\le 5/8$. Now, since the properties of $G$ which we have used also apply to finite index subgroups, this argument also applies to $K_1$. Hence we may construct a descending chain of characteristic finite index subgroups
\[\ldots\le K_i\le K_{i-1}\le\ldots\le K_2\le K_1 \le K_0=G\]
where, for every $i \in \N$, $\dc(K_{i-1}/K_i)\le 5/8$. Moreover $(G/K_i)/(K_{i-1}/K_i)=G/K_{i-1}$ and so, from (\ref{gallagher}),
$$\dc(G/K_i)\le \dc(G/K_{i-1})\cdot\dc(K_{i-1}/K_i)\le 5/8\cdot \dc(G/K_{i-1}).$$
By induction $\dc(G/K_i)\le (5/8)^i$ and so, by Proposition \ref{subexpprop}, for any finite generating set $X$ of $G$, we have that $\CR_X(G)\le \dc(G/K_i)\le (5/8)^i$. Since this holds for every $i \in \N$, we obtain that $\CR_X(G)=0$. 
\end{proof}

\begin{cor}
Conjecture \ref{maincon} is true for all finitely generated, virtually nilpotent groups, or equivalently, all groups of polynomial growth.
\end{cor}

\subsection{Virtually Abelian groups}

The goal of this section is to prove:

\begin{thm} \label{indx}
	Let	$G$ be a finitely generated, virtually abelian group, and $X$, $Y$ be finite generating sets for $G$. Then $\CR_X(G) = \CR_Y(G)$.
\end{thm}

It will be useful to have the following shorthand:

\begin{defn}
Let $G$ be generated by the finite set $X$. A subset, $S$, of $G$ is {\em generic} if $\limsup_{n\rightarrow \infty}\frac{|S\cap \B_{G, X}(n)|}{|\B_{G, X}(n)|} = 1$, and {\em negligible} if the limit is $0$. 
\end{defn}

 Given a group, $G$, with finite generating set $X$, 
a finitely generated subgroup, $H$, of $G$ is said to be {\em undistorted} if any word metric on $H$ is bi-Lipschitz equivalent to any word metric on $G$, when restricted to $H$. This makes sense since any two finite generating sets on a group induce bi-Lipschitz equivalent word metrics.

It is easy to see that a finite index subgroup is always undistorted, and that a subgroup $H$ is undistorted if and only if it has an undistorted subgroup of finite index. Retracts are also undistorted (recall that a retract of $G$ is the image of an endomorphism $\rho: G \to G$ such that $\rho^2 = \rho$). 

We now collect the following facts:
\begin{prop} \label{basicsvab}
	Suppose that $G$ is a finitely generated virtually abelian group, with finite generating set $X$, having a subgroup of finite index isomorphic to $\Z^d$.
	\begin{enumerate}[(i)]
		\item Every subgroup of $G$ is both finitely generated and undistorted. 
		\item Let $H$ be an infinite subgroup of $G$. Let $T(n)=T_{H,X}(n)$ (for transversal) be the number of cosets of $H$ that have a representative in $\B_{G,X}(n)$. Then, 
		$$
		\lim_{n \to \infty} \frac{T(n)}{|\B_{G,X}(n)|} = 0.
		$$
	\end{enumerate}
\end{prop}

\begin{proof}
(i) Let $H \leq G$. It is well known that $H$ is finitely generated, as this fact is true in the case where $G$ is virtually polycyclic, which includes the finitely generated virtually nilpotent (and abelian) case.
		
		However, the fact that $H$ is undistorted is not true more generally, and follows from the fact that every subgroup of a finitely generated free abelian group has finite index in a direct summand. In our case, $H$ has a finite index subgroup which is a retract of a finite index subgroup of $G$, and is therefore undistorted in $G$. 
		
(ii) From above, $H$ is finitely generated and undistorted. Since $H$ is infinite, it must contain an element of infinite order, so there exists an $\epsilon > 0$ such that 
		$$
		|H \cap \B_{G,X}(n)| \geq \epsilon n.  
		$$ 
		More precisely, $|H \cap \B_{G,X}(n)|$ will have polynomial bounds of degree $e$, $d \geq e \geq 1$.
		
		Let $A, B, d$ be the constants in (\ref{eqn:pol_bounds}). Then
		$$
		B 2^d n^d \geq |\B_{G,X}(2n)| \geq T(n) |H \cap \B_{G,X}(n)| \geq T(n) \epsilon n.
		$$
		Hence, 
		$$
		0 \leq \lim_{n \to \infty} \frac{T(n)}{|\B_{G,X}(n)|} \leq \lim_{n \to \infty} \frac{B2^d n^{d-1}}{\epsilon A n^d} =0.
		$$
\end{proof}

From now on, we let $G$ be an infinite, finitely generated, virtually abelian group. Let $A$ be a normal, finite index, free abelian subgroup,  
and $B$ be the centraliser of $A$ in $G$. Note that $A$ is a subgroup of $B$, which therefore has finite index.   

\begin{prop}\label{prop:ccnegligible}
Let $G$ be a finitely generated, virtually abelian group and $X$ any finite generating set for $G$. Let $A$ be a normal, finite index, free abelian subgroup, and $B$ be the centraliser of $A$ in $G$. Then the set of minimal length $G$-conjugacy representatives in $G \setminus B$ is negligible.
\end{prop}
\begin{proof}
 Let $y \not\in B$ be an element of $G$ and denote by $C_{yA}(n)$ the number of conjugacy classes which have a representative in $\B_{G,X}(n) \cap yA$. Then we claim that 
	$$
	\lim_{n \to \infty} \frac{C_{yA}(n)}{|\B_{G,X}(n)|} = 0.
	$$
	For each conjugacy class with a representative in $\B_{G,X}(n)\cap yA$, choose a shortest such representative, and denote this set of representatives $Z = \{ y a_i \ : \ a_i \in A\}$. From these, extract the set $U = \{ a_i \}$, rewriting the $a_i$ as geodesics if required. Note that, for some fixed $k$ (the length of $y$), we have
	$$
	C_{yA}(n) = | Z \cap \B_{G,X}(n)| \leq | U \cap \B_{G,X}(n+k)|.
	$$
	
	Now let $M_y$ denote the automorphism of $A$ induced by conjugation with $y$, which we think of as a matrix. For any $a_1,a_2 \in A$ we have that:
	$$
	a_2^{-1} (ya_1) a_2 = y (a_1 + (I - M_y)a_2), 
	$$
	if we switch to an additive notation in $A$. Let $H$ be the image of $(I-M_y)$ in $A$, that is, $H=\langle [a,y] \ : a\in A \rangle$. Since $y \not\in B$, we can conclude that $H$ is a non-trivial subgroup of $A$ and is therefore infinite. Moreover, the elements of $U$ are all in distinct cosets of $H$. Hence, by Proposition \ref{basicsvab} part (ii), we may conclude that:
	$$
	\begin{array}{rcl}
	\frac{| Z \cap \B_{G,X}(n)|}{|\B_{G,X}(n)|} & \leq &  \frac{| U \cap \B_{G,X}(n+k)|}{|\B_{G,X}(n)|} \\ \\
	& = & \frac{T_{H,X}(n+k)}{|\B_{G,X}(n+k)|} \frac{|\B_{G,X}(n+k)|}{|\B_{G,X}(n)|} \to 0.
	\end{array}
	$$
	\end{proof}

Proposition \ref{prop:ccnegligible} shows that the only elements of $G$ that contribute to the conjugacy ratio are the elements of $B$. (The representative of a conjugacy class might not have a shortest representative in our particular coset $yA$, but varying $y$ we see that we have an overcount of the number of conjugacy classes in the complement of $B$, which nonetheless gives $0$.)

Thus the strategy for proving Theorem \ref{indx} is the following. First note that each element of $B$ has finite conjugacy class in $G$. We split the elements in $B$ into those which centralise elements from outside of $B$ and those whose centraliser is completely in $B$. Proposition \ref{prop:negligible} shows the former ones form a negligible set, and the latter ones a generic set of $B$ (Corollary \ref{cor:generic}); moreover, for the latter ones the size of the $G$-conjugacy class is the index of the $B$-centraliser, which is constant for elements in the same $A$-coset. Therefore, each coset (or, rather, conjugacy class of cosets) of $A$ contributes a fixed amount to the conjugacy ratio, which is algebraically determined. 

We use the notation $Z_K(g)$ for the $K$-centraliser of $g \in G$, that is, $Z_K(g)= \{ k \in K \ : \ k^{-1} g k = g\}$.
\begin{lem}
	Let $x \in G$. Then $Z_B(x) = Z_B(xa)$ for any $a \in A$. Moreover, $[G:Z_B(x)] < \infty$ if $x\in B$. 
\end{lem}

\begin{prop}\label{prop:negligible}
	The set $\bigcup_{y \not\in B} Z_B(y)$ is a finite union of infinite index subgroups of $G$. Hence this set is negligible with respect to any finite generating set.
\end{prop}
\begin{proof}
	Since $Z_B(y)=Z_B(ya)$ for any $a\in A$, this is a finite union. So it is enough to show that each $Z_B(y)$ has infinite index. 
	
	In fact, it is sufficient to show that $Z_A(y) = Z_B(y) \cap A$ is an infinite index subgroup of $A$. However, $Z_A(y)$, is a pure subgroup of $A$; that is, if $a^m \in Z_A(y)$ and $m \neq 0$ then $a \in Z_A(y)$. This implies that $Z_A(y)$ is a direct summand of $A$. But since $y \not\in B$, this direct summand cannot be the whole of $A$ and is therefore an infinite index subgroup of $A$ as required. 
\end{proof}

\begin{cor}\label{cor:generic}
	There is a generic set of elements of $B$ (with respect to any generating set) whose centraliser lies entirely in $B$.
\end{cor}
\begin{proof}
	If for some $b \in B$ there exists $t\notin B$ such that $[t,b]=1$, then $b \in Z_B(t) \subset \bigcup_{y \not\in B} Z_B(y)$, which is negligible by Proposition \ref{prop:negligible}.
\end{proof}

\noindent {\em Proof of Theorem~\ref{indx}}: For each $r$, let $A_r$ be the elements $b\in B$ for which $Z_B(b)$ has index $r$ in $G$ (and therefore conjugacy class size $r$ in $G$), and let $\mathcal{N}=\{b\in B\ : Z_B(b) \not \subset B\}$, that is, $\mathcal{N}$ is the set of elements of $B$ whose centraliser does not fully lie in $B$. Then $\mathcal{N} = \bigcup_{y \not\in B} Z_B(y)$ and so by Corollary \ref{cor:generic} it is a negligible set. 

Since $A\leq Z_B(b)\leq G$ for any $b\in B$ and $A$ has finite index in $G$, there are only finitely many values for the index of $Z_B(b)$ in $G$, and thus finitely many $r$ for which $A_r$ is non-empty. Moreover, since $Z_B(y) = Z_B(ya)$ for any $y\in B\setminus A$ and $a\in A$, if $y\in A_r$, then $ya\in A_r$, so each non-empty $A_r$ is a union of $A$-cosets and thus 
\begin{equation}\label{eqn:coset}
\lim_{n \to \infty} \frac{|A_r \cap \B_{G,X}(n)|}{|\B_{G,X}(n)|} = \delta,
\end{equation}
where $\delta$ is $1/[G:A]$ times the number of $A$-cosets in $A_r$, so is independent of $X$. 


It is easy to see that there is an integer, $k$, such that if two elements of $B$ are conjugate in $G$, then they are conjugate by an element of length at most $k$;  the same holds for $A_r$ as $A_r \subset B$. Moreover, since $B$ is normal in $G$, it is easy to see that $G$ acts on $A_r$ by conjugation; $G$ acts by conjugation on $\mathcal{N}$, and hence on $A_r \setminus \mathcal{N}$, as well. 

Let $C_n$ be the number of conjugacy classes of $G$ which meet $\B_{G,X}(n)$ and are contained in $A_r\setminus \mathcal{N}$. Then, 
\begin{equation}\label{eqn:cArN}
| (A_r \setminus\mathcal{N}) \cap \B_{G,X}(n)| \leq r C_n \leq  |(A_r \setminus \mathcal{N})\cap \B_{G,X}(n+2k)|.
\end{equation}

The first inequality comes from the fact that each element of $A_r \setminus \mathcal{N}$ has $r$ conjugates in $G$, and the second from the fact that each of the conjugates can be obtained from a conjugator of length at most $k$.  

Now
$$
 \frac{C_n}{|\B_{G,X}(n)|} \leq \frac{|C_{G,X}(n) \cap A_r|}{|\B_{G,X}(n)|} \leq  \frac{C_n}{|\B_{G,X}(n)|} +  \frac{|\mathcal{N} \cap \B_{G,X}(n)|}{|\B_{G,X}(n)|}, 
$$
and by (\ref{eqn:coset}), (\ref{eqn:cArN}) and Corollary \ref{cor:generic}
$$
\lim_{n \to \infty} \frac{C_n}{|\B_{G,X}(n)|} = \lim_{n \to \infty} \left(\frac{C_n}{|\B_{G,X}(n)|} +  \frac{|\mathcal{N} \cap \B_{G,X}(n)|}{|\B_{G,X}(n)|}\right) = \frac{\delta}{r},
$$
so we get $$\lim_{n \to \infty}\frac{|C_{G,X}(n) \cap A_r|}{|\B_{G,X}(n)|}=\frac{\delta}{r}.$$

Hence the number of conjugacy classes of $G$ that meet $A_r$ is independent of the generating set. Summing over the finitely many $r$ gives the result.\qed

%

\medskip

\begin{rem}
	The same ideas as those just presented can be used to show that, if $G$ is a finitely generated, virtually abelian group, and $X$ is any finite generating set, then $\CR_X(G) = \inf_{N \unlhd_f G} \CR(G/N)$. That is, the conjugacy ratio is equal to the infimum of conjugacy ratios of the finite quotients. Hence, if one were to measure the conjugacy ratio using invariant means, one would get the same numerical value. Unpublished results indicate that this is the same as the degree of commutativity.
	
	For similar reasons, the same is true whenever $G$ is a finitely generated virtually nilpotent group, the virtually abelian case being the key one. 
\end{rem}


\section{Results for other families of groups}\label{generalresults}
\subsection{Hyperbolic groups}
In this section we prove Conjecture \ref{maincon} for non-elementary hyperbolic groups. 

\medskip

We will write $f(n) \sim g(n)$ to mean $f(n)/g(n) \rightarrow 1$ as $n \rightarrow \infty$.

\begin{thm}\label{thm:hyperbolic}
Let $G$ be a non-elementary hyperbolic group. Then $\CR_X(G)= 0$ for any finite generating set $X$.
\end{thm}
\begin{proof}
Let $G$ be a non-elementary hyperbolic group with finite generating set $X$. Then by a result of Coornaert (see \cite{c93}) there are positive constants $A_0$,$B_0$, and integer $n_0$, such that for all $n\geq n_0$
\begin{equation}\label{eq:growthbounds}
A_0e^{n \grate}\leq |\B_{G,X}(n)| \leq B_0e^{n\grate},
\end{equation}
where $\grate=Exp_X(G)$. 

By Theorem 1.2 in \cite{AC17}, there are positive constants $A_1, B_1$ and $n_1$ such that 
\begin{equation}\label{eq:conjgrowthbounds}
A_1\frac{e^{n\grate }}{n} \leq  |C_{G,X}(n)| \leq B_1 \frac{e^{n\grate}}{n} 
\end{equation}
for all $n \geq n_1$. Thus from (\ref{eq:growthbounds}) and (\ref{eq:conjgrowthbounds}) we get $$\frac{|C_{G,X}(n)|}{|\B_{G,X}(n)|}\le \frac{B_1}{A_0 n}$$ for all $n\geq \max(n_0, n_1)$, and by taking the limit we obtain that $\CR_X(G)=0$.
\end{proof}

\subsection{The lamplighter group} 

We follow the notation in \cite{VM17}. Let $I$ be a non-empty set. For $\eta\in \bigoplus_{i\in  I}G$ we write $\eta(i)$ for the $i^\textup{th}$ component of $\eta$, and if moreover $I$ is a group and $x\in I$, we define $\eta^x\in\bigoplus_{i\in  I}G$ by $\eta^x(i)=\eta(x^{-1}i)$, and say that $\eta^x$ is the {\em left translate} of $\eta$ by $x$.
\begin{defn}
Consider groups $H$ and $L$ with symmetric generating sets $A$ and $B$, and neutral elements $e$ and $e'$, respectively. The {\em wreath product of }$G$ {\em by }$L$, written $G\wr L$, is defined as
\[
H\wr L:=\bigoplus_{i\in L} H \rtimes L,
\]
where for $(\eta,m),(\theta,n)\in H\wr L$, $(\eta,m)(\theta,n)=(\eta\theta^m,mn)$.

\end{defn}
 For $h\in H$, let $\vec{h}\in \bigoplus_{i\in L}H$ be such that $\vec{h}(e')=h$ and $\vec{h}(i)=e$ for $i\neq e'$. Then
\begin{equation}\label{vecY}
 X:=\{(\vec{e},a)\,:\,a\in A\}\cup\{(\vec{b},e')\,:\,b\in B\}.
\end{equation}
 generates $H\wr L$.

For the lamplighter group $G=C_2 \wr \mathbb{Z}$ we let $A:=\{a\}$, where $a$ is the non-trivial element of $C_2$, and let $B$ be the standard generating set of $\mathbb{Z}$.

\begin{thm}\label{thm:lamplighter}
Let $G$ be the lamplighter group, that is, the wreath product $C_2\wr\mathbb{Z}$. Then $\CR_X(G)= 0$ for the standard generating set $X$.
\end{thm}
\begin{proof}
The statement follows immediately from \cite[Example 5.0.3]{VM17}, where it is shown that $|C^s_{G,X}(n)|\sim \frac{2}{n}{\left(\frac{1+\sqrt{5}}{2}\right)}^{n}$, and the fact that $|S_{G,X}(n)|\sim {\left(\frac{1+\sqrt{5}}{2}\right)}^{n}$ by  \cite{Par92}.
\end{proof}

\subsection{Right-Angled Artin Groups}

Let $\Gamma=(V,E)$ be a simple graph (i.e. a non-oriented graph without loops or multiple edges) with vertex set $V$ and edge set $E$.
 For each vertex $v$ of $\Gamma$, let
$G_v$ be a group. The {\em graph product of the groups} $G_v$ {\em with respect to}
$\Gamma$ is defined to be the quotient of their free product by the normal
closure of the relators $[g_v,g_w]$ for all $g_v \in G_v$, $g_w \in G_w$
for which $\{v,w\}$ is an edge of $\Gamma$. Here we consider right-angled Artin groups (RAAGs), which are graph products with all $G_v = \mathbb{Z}$, and denote by $(G_V, X_V)$ the RAAG based on the graph $\Gamma$ with generating set $X_V$ (in bijection to $V$).

Conjugacy representatives in a RAAG come, to a large extent, from taking one word out of each cyclic permutation class, so we first establish the asymptotics of the language of cyclic representatives in a rather general setting. 

\begin{exmp} In a free group on the free generating basis, counting the conjugacy classes with a minimal representative of length $n$ is equivalent to counting the number of cyclically reduced words of length $n$, up to cyclic permutation.
\end{exmp}

\subsubsection{Cyclic representatives of languages}
We follow the notation in \cite[Section 2.3]{CHM17}.
Let $L$ be a language over a finite alphabet $X$, that is, $L\subseteq X^*$, and let $L(n)$ denote the set of words of length $\leq n$ in $L$.
For $n \geq 1, n\in \mathbb{N}$, let $L^n:= \{w^n \mid w\in L\}$ and $\sqrt[n]{L}=\{v \mid v^n \in L\}$. Define $\prim(L):=\{w \in L \mid \nexists k>1, v \in L \textrm{ such that } v^k=w \}$ to be the language of \emph{primitive words} in $L$.  

Suppose $L$ is closed under cyclic permutations; then we construct a language $\cycrep(L)$ of cyclic representatives of $L$ out of the words $w_c$, where $w_c$ the word that is least lexicographically among all cyclic permutations of $w$, for $w \in L$:
 $$\cycrep(L):=\{w_c \mid w \in L\}.$$ 

\begin{prop}[see also Lemma 2.10 (4), \cite{CHM17}]\label{decycle}
Let $L$ be an exponential growth language closed under cyclic permutations. Furthermore assume that $L^k \subseteq L$ and $\sqrt[k]{L} \subseteq L$ for all $k \geq 1$.
Then $$\lim_{n \rightarrow \infty} \frac{|\cycrep(L)(n)|}{|L(n)|}=0.$$ 
\end{prop}

\begin{proof}
For simplicity of notation let $a(n):=|L^s(n)|$, $p(n):=|\prim(L)^s(n)|$ and $c(n):=|\cycrep(L)^s(n)|$, that is, we consider the numbers of words of length exactly $n$ in each language. 

Write $L$ as $L=\bigcup_{k\geq 1}\prim^k(L),$ and notice that the number of cyclic representatives of length $n$ in $\prim(L)$ is $p(n)/n$, and the number of cyclic representatives of length $nk$ in $\prim^k(L)$ is also $p(n)/n$. Thus $a(n)=\sum_{d/n} p(d)$ and $c(n)=\sum_{d/n} \frac{p(d)}{d}$. Let $\mu(n)$ and $\phi(n)$ be the standard number theoretic M\"obius and Euler functions.
Then by M\"obius inversion $p(n)=\sum_{d/n} \mu(\frac{n}{d})a(n)$ and so $$c(n)=\sum_{d/n} \frac{\sum_{l/(n/d)} \frac{\mu(l)}{l}a(d)}{d}=\sum_{d/n} a(d) \frac{\phi(n/d)}{n},$$ which follows from $\sum_{d/n}\phi(d)=n$ and $\sum_{d/n}\frac{\mu(d)}{d}=\frac{\phi(n)}{n}$.

Since $a(n)$ is exponential, only the last term in the sum above is of the same magnitude as $a(n)$, so
\begin{equation}\label{sphericalaprox}
c(n) \sim \frac{a(n)}{n} \implies \lim_{n \rightarrow \infty} \frac{|\cycrep(L)^s(n)|}{|L^s(n)|}=0.
\end{equation}
By Stolz-Ces\`aro we obtain the result.
\end{proof}

%

\subsubsection{Conjugacy representatives in RAAGs}
We first establish a result about the conjugacy ratio of direct products.
\begin{lem}\label{directprod}
Let $H$ and $K$ be two groups with finite generating sets $X$ and $Y$, respectively. If either (i) $\CR_X(H)=\CR_Y(K)=0$ or, (ii) $\CR_X(H)=0$ and $Exp_X(H) > Exp_Y(K)$, then $\CR_{X\cup Y}(H\times K)=0$.
\end{lem}

\begin{proof}
	
We calculate the conjugacy ratio with respect to balls in $H \times K$.	
To do this we use balls in $H$ and spheres in $K$. 
Let $a_n:=|C_{H,X}(n)|$, $b_n:=|\B_{H,X}(n)|$, $t_n:=|C^s_{K,Y}(n)|$ and $s_n:=|S_{K,Y}(n)|$. Then $$\CR_{X\cup Y}(H\times K)=\limsup_{n\rightarrow \infty}\frac{\sum_{i=0}^n a_i t_{n-i}}{\sum_{i=0}^n b_i s_{n-i}}.$$

If $\CR_X(H)= \CR_Y(K)=0$, then by Proposition \ref{convolution} (putting $t_n=\widehat{c_n}, s_n=\widehat{d_n}$) we get that $\CR_{X\cup Y}(H\times K)=0$. Similarly, if $\CR_X(H)=0$ and $Exp_X(H) > Exp_Y(K)$ then Proposition \ref{convolution2} (putting $c_n= t_n, d_n=s_n$) states that this limit is zero, so $\CR_{X\cup Y}(H\times K)=0$.
\end{proof}
 
Since RAAGs interpolate between free and free abelian groups, the presence of commutativity does not allow us to simply consider cyclically reduced words up to permutation, as in free groups. We need to single out the words for which taking cyclic representatives produces conjugacy representatives, and use Crisp, Godelle and Wiest's approach from \cite{CGW}, which was further developed in \cite{CHM17}. 

\begin{defn}[Def 2.19, \cite{CGW}]
Let $V=\{a_1, \dots, a_N\}$ and set the total order $a_1<a_1^{-1}<a_2<a_2^{-1}< \dots$. 
 A cyclically reduced word $w$ is in \emph{cyclic normal form} if it is in the shortlex language $\sphl(G_V,X_V)$ of $G_V$ with respect to $X_V$ and all its cyclic conjugates are in $\sphl(G_V,X_V)$ as well. 
 \end{defn}

Not all elements posses a cyclic normal form. For example, if $[a_1,a_2]=1$, the word $a_1a_2$ is in $\sphl(G_V,X_V)$, but its cyclic permutation $a_2a_1$ is not. To deal with this situation, \cite{CGW} divides the words over $X_V$ into \emph{split} and \emph{non-split}.
 
\begin{defn}[Definition 2.13, \cite{CGW}]
Let $w$ be a cyclically reduced word over $X_V$ and denote by $\Delta(w)$ the full subgraph spanned by $\supp(w)$. Let $\overline{\Delta(w)}$ be the graph complement of $\Delta(w)$.
\begin{enumerate}
\item[(i)] The word $w$ is \emph{split} if $\overline{\Delta(w)}$ is disconnected, which amounts to being able to write $w$ as a product of commuting subwords (or blocks).

\item[(ii)] The word $w$ is \emph{non-split} if $\overline{\Delta(w)}$ is connected. 

\item[(iii)] Let $\cycsl(G_V,X_V)$ denote the set of all non-trivial cyclic normal forms corresponding to non-split words in $G_V$. 
\end{enumerate}
\end{defn} 

We say that a group element is \emph{non-split (split)} if it can be represented by a cyclically reduced word which is non-split (split). 

 \begin{prop}[Prop. 2.21, \cite{CGW}]\label{CNFnonsplit}
Two cyclic normal forms represent conjugate elements if and only if they are equal up to a cyclic permutation. 
 \end{prop}
 

\begin{prop}[Remark 2.14, \cite{CGW}]\label{splitconj}
Let $w$ and $v$ be two cyclically reduced split words. Then they are conjugate if and only if $\Delta(w)=\Delta(v)$ and the words corresponding to the commuting blocks are conjugate, respectively.
\end{prop}

\begin{lem}\cite{CHM17} \label{twoconditions}
Let $\cycsl(G_V,X_V)$ be the set of cyclic normal forms in $G_V$.
The following hold:

(1) $\cycsl^k(G_V,X_V)\subseteq \cycsl(G_V,X_V)$ for all $k\geq 1$, and 

(2) $\cycsl(G_V,X_V))$ is closed under cyclic permutations.

\end{lem}

\begin{thm}\label{thm:RAAG}
Let $G=(G_V, X_V)$ be a right-angled Artin group (RAAG) based on a graph $\Gamma=(V,E)$ with generating set $X_V$. Then $\CR_{X_V}(G)= 0$ unless $G$ is free abelian, in which case $\CR_{X_V}(G)=1$.
\end{thm}

\begin{proof}
We use induction on the number of vertices. Let $n:=|V|$. The result is trivial for $n=1$. If $G$ is a direct product, then we get $\CR(G)=0$ if at least one of the factors has $\CR=0$; this follows from Lemma \ref{directprod}(i) if both factors have $\CR=0$ and from Lemma \ref{directprod}(ii) if, say, the first factor has $\CR=0$, as the second is by induction free abelian, and of strictly smaller growth rate than the first. We get $\CR(G)=1$ when each factor is free abelian. 

So suppose $G$ is not a direct product. We split the conjugacy classes $C_{G_V,X_V}$ of $G$ into two types: those which have a shortest length representative with support $X_U$, where $U \subsetneq V$, and denote these by $C_{G_V,\leq X_V}$, and those which have a shortest length representative with support exactly $X_V$, and denote these by $C_{G_V,=X_V}$. By Propositions \ref{CNFnonsplit} and \ref{splitconj}, this is well defined. Moreover, by Propositions \ref{CNFnonsplit} and \ref{splitconj}, two cyclically reduced words $w_1$, $w_2$ with support $X_U$ are conjugate in $G_V$ if and only if they are conjugate in $G_U$ (note that if a word $w\in \cycsl(G_V,X_V)\cap X^*_{U}$, where $U\subsetneq V$, then $w\in \cycsl(G_U,X_U)$). 

Thus we can write $C_{G_V,\leq X_V} \subseteq \bigcup_{U \subsetneq V} C_{G_U,X_U}$ and express the above as:
\begin{equation}\label{decomposition}
C_{G_V,X_V} \subseteq \bigcup_{U \subsetneq V}C_{G_U,X_U}\bigcup C_{G_V,=X_V}.
\end{equation}
Then (\ref{decomposition}) implies that 
\begin{equation}\label{frac}
\frac{|C_{G_V,X_V}(n)|}{|\B_{G_V,X_V}(n)|} \leq \frac{\left(\sum_{X_U, U \subsetneq V} |C_{G_U,X_U}(n)|\right) + |C_{G_V,=X_V}(n)|}{|\B_{G_V,X_V}(n)|}.
\end{equation}
Now for $U \subsetneq V$ $$\frac{|C_{G_U,X_U}(n)|}{|\B_{G_V, X_V}(n)|}=\frac{|C_{G_U, X_U}(n)|}{|\B_{G_U, X_U}(n)|}\frac{|\B_{G_U, X_U}(n)|}{|\B_{G_VX_V}(n)|}, $$
so $$\limsup_{n\rightarrow \infty}\frac{|C_{G_U,X_U}(n)|}{|\B_{G_V,X_V}(n)|} \leq \CR_{X_U}(G_U)\limsup_{n\rightarrow \infty}\frac{|\B_{G_U,X_U}(n)|}{|\B_{G_V,X_V}(n)|}.$$
The right hand side is equal to $0$ since either (i) $\CR_{X_U}(G_U)=0$ by induction, or (ii) $G_U$ is free abelian (so of polynomial growth); if (ii), since $G$ itself if not a direct product by assumption, it is of exponential growth, and the last fraction is $0$.

It remains to find $\limsup_{n \rightarrow \infty} \frac{|C_{G_V,=X_V}(n)|}{|\B_{G_V,X_V}(n)|}$, the second part of the right hand side of (\ref{frac}). Since $G$ is not a direct product, all conjugacy representatives with support exactly $X_V$ are non-split, so it suffices to consider cyclic normal forms up to cyclic permutations, that is $$\frac{|C_{G_V,=X_V}(n)|}{|\B_{G_V,X_V}(n)|}\leq\frac{|\cycrep(\cycsl(G_V,X_V)(n)|}{|\sphl(G,X_V)(n)|}=$$
$$\frac{|\cycrep(\cycsl(G_V,X_V)(n)|}{|\cycsl(G,X_V)(n)|}\frac{|\cycsl(G_V,X_V)(n)|}{|\sphl(G,X_V)(n)|},$$ 
and by Proposition \ref{decycle} applied to the language $\cycsl(G_V,X_V)$ (which satisfies the hypothesis of Proposition \ref{decycle} by Lemma \ref{twoconditions}) 
$$\lim_{n\rightarrow \infty}\frac{|\cycrep(\cycsl(G_V,X_V)(n)|}{|\cycsl(G,X_V)(n)|}=0.$$
This proves the result.
\end{proof}

\section{Reflections and open questions}

Our results on the conjugacy ratio values are essentially identical to those on the degree of commutativity in \cite{dcA, Cox3, MV17}. That is, the two quantities are equal for all the classes of groups we studied. However, we could not establish a direct general link between them. 

\begin{quR}
Is the limsup in the definition of the conjugacy ratio a limit?
\end{quR}

\begin{quR}
What are the groups for which $\dc_X(G) \leq \CR_X(G)$ (or vice versa)? They are equal in the virtually nilpotent case, in the hyperbolic group case and in many more. 
\end{quR}

As is the case for the degree of commutativity, we do not know whether the conjugacy ratio might be influenced by a change of generators.
\begin{quR}
Does there exist a group $G$ with finite generating sets $X$ and $Y$ such that $\CR_X(G)\neq \CR_Y(G)$?
\end{quR}

Finally, it would be interesting to unify the proofs confirming our conjecture for larger classes of groups, such as all groups of exponential growth, for example.


\bibliographystyle{amsalpha}
\def\cprime{$'$}
\providecommand{\bysame}{\leavevmode\hbox to3em{\hrulefill}\thinspace}
\providecommand{\MR}{\relax\ifhmode\unskip\space\fi MR }
\providecommand{\MRhref}[2]{%
  \href{http://www.ams.org/mathscinet-getitem?mr=#1}{#2}
}
\providecommand{\href}[2]{#2}
\end{document}